\documentclass[a4paper,12pt]{article}
\usepackage{amsmath}
\usepackage{amsfonts}
\usepackage{appendix}
\usepackage{amsthm}
\usepackage{amssymb}
\usepackage{enumitem}
\usepackage{yhmath}
\usepackage{fullpage}
\usepackage[all]{xy}
\usepackage{marginnote}
\usepackage{pgf,tikz,pgfplots}
\pgfplotsset{compat=1.15}
\usepackage{mathrsfs}
\usepackage{color}
\usetikzlibrary{arrows}
\numberwithin{equation}{section}
\usepackage{graphicx}
\usepackage{caption}
\usepackage{subcaption}
\usepackage[english]{babel}
\usepackage{tocloft}
\usepackage[colorlinks, linkcolor=blue, citecolor=blue, urlcolor=blue]{hyperref}
\usepackage{latexsym}
\usepackage{mathrsfs}
\usepackage{pifont}

\newcommand{\C}{\mathbb{C}}
\newcommand{\R}{\mathbb{R}}
\newcommand{\Z}{\mathbb{Z}}
\newcommand{\N}{\mathbb{N}}

\newcommand{\qbinom}[2]{\begin{bmatrix}#1\\#2\end{bmatrix}}
\newcommand{\la}{\lambda}

\newtheoremstyle{rem}{3pt}{3pt}{}{}% <Space above><Space below><Body font> <Indent amount>
{\bfseries}{.}{.5em}{}% <Theorem head font><Punctuation after theorem head><Space after theorem headi><Theorem head spec>

\newtheorem{theo}{Theorem}[section]
\newtheorem*{theo*}{Theorem}
\newtheorem{lemm}[theo]{Lemma}

\newtheorem{prop}[theo]{Proposition}

\newtheorem{coro}[theo]{Corollary}

\theoremstyle{rem}
\newtheorem{remax}[theo]{Remark}

\newenvironment{rema}
  {\pushQED{\qed}\remax}
  {\popQED\endremax}

\theoremstyle{definition}

\newtheorem*{term*}{Notation/Terminology}

\usepackage{geometry}
\title{Little $q$-Jacobi polynomials and symmetry breaking operators for $U_q(sl_2)$}

\author{Q. Labriet\footnote{ Centre de Recherches Math\'ematiques, Universit\'e de Montr\'eal, P.O. Box 6128, Montréal, Canada. \emph{email adress}: quentin.labriet@umontreal.ca}, L. Poulain d'Andecy\footnote{Laboratoire de math\'ematiques de Reims LMR, UMR 9008, Universit\'e de Reims Champagne-Ardenne, Moulin de la Housse BP 1039, 51100 Reims, France. \emph{email adress}: loic.poulain-dandecy@univ-reims.fr}}

\date{}

\begin{document}
\maketitle

\begin{abstract}
This paper presents explicit formulas for intertwining operators of the quantum group $U_q(sl_2)$ acting on tensor products of Verma modules. We express a first set of intertwining operators (the \emph{holographic operators}) in terms of the little $q$-Jacobi polynomials, and we obtain for the dual set (the \emph{symmetry breaking operators}) a $q$-deformation of the Rankin--Cohen operators. The Verma modules are realised on polynomial spaces and, interestingly, we find along the way the need to work with non-commuting variables. Explicit connections are given with the Clebsch--Gordan coefficients of $U_q(sl_2)$ expressed with the $q$-Hahn polynomials.
\end{abstract}
\section{Introduction}

The main purpose of this paper is to study the intertwining operators associated to the following decomposition of representations of the quantum group $U_q(sl_2)$:
\[V_{\lambda}\otimes V_{\lambda'}\cong \bigoplus_{n\geq 0} V_{\lambda+\lambda'+2n}\,,\]
where $V_{\lambda}$ is the lowest weight Verma module of weight $\lambda$, and to show the appearance of orthogonal polynomials from the $q$-Askey scheme and of $q$-differential operators (operators involving the $q$-derivative). In a nutshell, we will show how the family, on $n$, of intertwining operators 
\begin{equation}\label{introholographic}\Psi^{\lambda,\lambda'}_n\ :\ \ V_{\lambda+\lambda'+2n}\to V_{\lambda}\otimes V_{\lambda'},
\end{equation}
involves the little $q$-Jacobi polynomials, while the dual intertwining operators
\begin{equation}\label{introsymmetrybreaking}\operatorname{q-RC}^{\lambda,\lambda'}_n\ :\ \ (V_{\lambda}\otimes V_{\lambda'})^{\star} \to V_{\lambda+\lambda'+2n}^{\star}, \end{equation}
are given by $q$-deformations of Rankin--Cohen brackets. In some context, the operators in \eqref{introsymmetrybreaking} are called \emph{symmetry breaking} operators while those in \eqref{introholographic} are called \emph{holographic} operators \cite{KobaPevznerIII}. Their link with little $q$-Jacobi polynomials generalises the appearance in the same context for $sl_2$ of the classical Jacobi polynomials and of the Rankin--Cohen operators \cite{LabrietPoulain23,LabrietPoulain24}. However, we will see that the $q$-deformed situation involves new features, not least among them the necessity of working with non-commuting variables.

To put our results into perspective, let us briefly recall the classical $sl_2$ situation. First of all, Verma modules are realised on the space of polynomials in one variable. Therefore, the intertwining operator (\ref{introholographic}) becomes a map between polynomial spaces and it turns out to be possible to give it as follows (see \cite{LabrietPoulain23}):
\begin{equation}\label{introPsi}
\begin{array}{lrcl}\Psi^{\lambda,\lambda'}_n\ :\ & \C[z] & \to & \C[x]\otimes \C[y]\cong \C[t,tv] \\[0.5em]
 & z^k & \mapsto & t^{k+n}P_n^{(\lambda-1,\lambda'-1)}(v)\end{array}\ , \end{equation}
where the polynomials $P_n^{(\lambda-1,\lambda'-1)}(v)$ are the classical Jacobi polynomials, and $t$ and $v$ are given by $t=x+y$ and $v=\frac{y-x}{x+y}$. The main feature of this change of variables is that the polynomials $t^nP_n^{(\lambda-1,\lambda'-1)}(v)$, $n\geq0$, are the lowest weight vectors of the representations $V_{\lambda+\lambda'+2n}$ appearing in the decomposition of the tensor product, while for a fixed $n$, the diagonal action of $U(sl_2)$ acts on the variable $t$ on the set of polynomials $\{t^{k+n}P_n^{(\lambda-1,\lambda'-1)}(v)\}_{k\geq0}$.

Then, passing to the duals, the intertwining operators (\ref{introsymmetrybreaking}) are realised as bidifferential operators (see \cite{LabrietPoulain24}):
\begin{equation}\label{introRC}
\begin{array}{lrcl}RC^{\lambda,\lambda'}_n\ :\ & \C[x]\otimes \C[y] & \to & \C[z] \\[0.5em]
 & f\otimes g & \mapsto & RC^{\lambda,\lambda'}_n(f,g)\end{array}\ .\end{equation}
This can be done through a suitable identification of $V_{\lambda}^{\star}$ with a polynomial space and the resulting operators $RC^{\lambda,\lambda'}_n$ turn out to coincide with the so-called Rankin--Cohen brackets from the world of modular forms \cite{Cohen75}. Roughly speaking, the connections between (\ref{introPsi}) and (\ref{introRC}) can be expressed by the fact that the symbols of the bidifferential operators $RC^{\lambda,\lambda'}_n$ are the two variable polynomial in $(x,y)$ obtained from the Jacobi polynomials in (\ref{introPsi}).

\vskip .2cm
The present paper generalises the above situation to the case of the quantum group $U_q(sl_2)$. On the one hand, we find that the role of the classical Jacobi polynomials is played by the little $q$-Jacobi polynomials, which are well-known $q$-deformations of the Jacobi polynomials in the $q$-version of the Askey scheme \cite{KLS10}. On the other hand, we find for the intertwining operators $RC^{\lambda,\lambda'}_n$ explicit $q$-deformations of the Rankin--Cohen brackets. We obtain formulas similar to the ones found in \cite{Ros} but the approach through the little $q$-Jacobi polynomials is different. Note that connections between the little $q$-Jacobi polynomials and $U_q(sl_2)$ are known, such as the fact that they appear as matrix elements (see \cite{Koo90} and references therein) or via representations of the Askey--Wilson algebra \cite{BMVZ20}. However the appearance of the little $q$-Jacobi polynomials presented here is different and seems to be new.

An interesting feature of the quantum setting appears immediately when one tries to generalise the change of variables in (\ref{introPsi}). Indeed in (\ref{introPsi}) all subrepresentations of the tensor product were obtained upon multiplying by the variable $t$ a given lowest weight vector. The multiplication by $t$ played the role, for all subrepresentations, of the raising operator moving from the lowest weight vector to vectors of higher weights. It turns out that in the quantum setting, such a feature can appear only if we impose for the variables $x$ and $y$ to be $q$-commuting instead of simply commuting. Therefore, all our study here will involve the so-called quantum plane in two variables instead of classical polynomials in two variables. 
%This was maybe to be expected in the quantum group setting but this appears very naturally in the present approach. 

The explicit decomposition of tensor products involves the Clebsch--Gordan coefficients, which are known for $U_q(sl_2)$ to be given in terms of $q$-Hahn polynomials (also polynomials in the $q$-version of the Askey scheme \cite{KLS10}). As a byproduct of our construction, we exhibit explicit connections between the little $q$-Jacobi polynomials and the $q$-Hahn polynomials, in the spirit for example of \cite{KVJ98,LabrietPoulain23}. We also find a realisation of the $q$-Hahn algebra in terms of the $q$-difference operator of the little $q$-Jacobi polynomials.

\vskip .2cm
The results presented here can be looked upon through the lens of the so-called F-method introduced by Kobayashi--Pevzner \cite{KobaPevznerI,KobaPevznerII}. Using this method for holomorphic discrete series representations of $\operatorname{SL}_2(\R)$, the symbols of Rankin--Cohen bidifferential operators are related to solutions of the hypergeometric differential equation and thus to Jacobi polynomials. The heart of the F-method relies on an algebraic Fourier transform defined as a morphism between two Weyl algebras. For example, in \cite{LabrietPoulain24} the algebraic Fourier transform was realised explicitly using the adjonction for the Fischer inner product on Verma modules. The F-method was used successfully for other Lie groups and other class of representations (see \cite{KKP16,KobaPevznerI,KobaPevznerII,KobaSpehI, KobaSpehII, KuboOrsted24,P-V23} for example) but as far as we know, it is a open question to extend it to the quantum group setting.

 Our results show that several features of the F-method may also exist in the quantum setting. First, the intertwining operators in (\ref{introholographic}) are related to solutions of a $q$-version of the hypergeometric differential equation, and that is how the little $q$-Jacobi polynomials appear. Then, to dualise the map in (\ref{introholographic}) into the intertwining operators in (\ref{introsymmetrybreaking}) we identify the contragredient representation spaces with the original ones via a $q$-deformation of the Fischer inner product. This could be seen as an instance of a $q$-version of the algebraic Fourier transform used in the F-method. In some sense the connection between the $q$-Rankin--Cohen brackets and the little $q$-Jacobi polynomials is made through a $q$-analogue of the notion of symbol of a differential operator (see Section \ref{sec5}).
Thus, the present results may be seen as a first step and a fundamental example towards the formulation and the understanding of an analog of the F-method in the quantum group setting.

\vskip .2cm
As was done in \cite{LabrietPoulain23,LabrietPoulain24}, once the two-fold tensor products are well understood, one may be interested in three-fold (and in general $n$-fold) tensor products. Indeed, the Racah problem for three-fold tensor products of $U_q(sl_2)$-representations involves the famous Askey--Wilson polynomials (see for example \cite{CFGPRV,Zhe91}). Based on the present study, interesting connections between little $q$-Jacobi polynomials, $q$-Rankin--Cohen brackets, Askey--Wilson polynomials and their Askey--Wilson algebra can be expected, and we leave this for a future work.

\paragraph{Organisation of the paper.} In Section \ref{sec2}, we introduce the algebra $U_q(sl_2)$ and its representations and we perform the first calculations of the lowest weight vectors in a tensor product. Here is explained the necessity of the $q$-commuting variables $x$ and $y$. In Section \ref{sec3}, we achieve the (generalisation in the quantum setting) of the separation of variables and show how the little $q$-Jacobi polynomials enter the stage. In Section \ref{sec4}, we make the connections with the Clebsch--Gordan coefficients for $U_q(sl_2)$, and show explicitly the appearance of the $q$-Hahn algebra and the connections between little $q$-Jacobi polynomials and $q$-Hahn polynomials. Finally, in Section \ref{sec5} we express the intertwining operators as explicit $q$-deformations of the classical Rankin--Cohen brackets.

\paragraph{Acknowledgements} Quentin Labriet acknowledges a postdoctoral fellowship from Luc Vinet’s NSERC discovery grant. Both authors also wish to thank the CRM for its hospitality. The authors also thank an anonymous referee for their valuable suggestions.

\paragraph{Notations}
All through the paper, $q$ is non-zero complex number which is not a root of unity (most of the paper can also be read with $q$ an indeterminate). Only in the last section we need $q$ to be a strictly positive real number. We use the following convention for $q$-numbers for $x\in \C$
\begin{equation}
[x]=\frac{q^x-q^{-x}}{q-q^{-1}}.
\end{equation}
We also introduce the following Pochhammer like notation together with $q$-factorials and $q$-binomial coefficients (from left to right)
\begin{equation*}
[x]_n=\prod_{s=0}^{n-1} [x+s],\ \ [n] !=[1]_n=\prod_{s=1}^{n} [s],\ \ \text{and}\ \ \qbinom{n}{k}=\frac{[n]!}{[k]![n-k]!}.
\end{equation*}
For $n\in \N$, the basic hypergeometric function ${}_{n+1}\phi_n$ is defined, for $a_1,\cdots, a_{n+1}\in \C$ and $b_1,\cdots b_n\in \C\backslash \Z_{\leq 0}$, by 
\begin{equation}\label{eq:HypergeoDef}
    {}_{n+1}\phi_n\left(\begin{array}{c} a_1,\cdots, a_{n+1}\\ b_1,\cdots ,b_n\end{array};q,x\right)=\sum_{k=0}^\infty \frac{(a_1;q)_k\cdots (a_{n+1};q)_k}{(b_1;q)_k\cdots (b_n;q)_k}\frac{z^k}{(q;q)_k},
\end{equation}
where
\begin{equation*}
    (a;q)_k=\prod_{i=1}^k(1-aq^{i-1}). 
\end{equation*}
In this paper, we restrict ourselves to the case where one of the $a_i$ is a negative integer so that the series is a finite sum and we do not have any convergence issues. 

Finally, we will use the following notation for $q$-derivatives for polynomials $f\in\C[z]$ and formal power series $f\in\C[[z]]$
\begin{equation*}
D_q f(z)=\frac{f(z)-f(qz)}{(1-q)z}.
\end{equation*}

\section{Tensor products of $U_q(sl_2)$-representations}\label{sec2}

\subsection{The Hopf algebra $U_q(sl_2)$ and its Verma modules}

The algebra $U_q(sl_2)$ is generated by $K,E,F$ with $K$ invertible and defining relations:
\[KEK^{-1}=q^2E\,,\ \ \ \ \ KFK^{-1}=q^{-2}F\,,\ \ \ \ \ [E,F]=\frac{K-K^{-1}}{q-q^{-1}}\ .\]
The coproduct and the antipode are given by
\[\Delta(K)=K\otimes K\,,\ \ \ \ \Delta(E)=E\otimes 1 + K\otimes E\,,\ \ \ \ \ \Delta(F)=F\otimes K^{-1} + 1\otimes F,\]
\[S(K)=K^{-1}\,,\ \ \ \ \ \ \ S(E)=-K^{-1}E\,,\ \ \ \ \ \ \ S(F)=-FK\ .\]
The Casimir element of $U_q(sl_2)$ is the central element 
\begin{equation}\label{eq:Casimir}
C=(q-q^{-1})^2\left(FE+\frac{qK+q^{-1}K^{-1}}{(q-q^{-1})^2}\right).
\end{equation}

\paragraph{Verma modules.} For $\lambda\in \C$, we denote $V_{\lambda}$ the representation of $U_q(sl_2)$ on the vector space $\mathbb{C}[z]$ of polynomials in one variable, where the action of the generators is
\begin{equation}\label{Verma}
\begin{array}{rcl}
K^{\pm 1} P(z)&=&q^{\pm\lambda}P(q^{\pm 2} z),\\[0.3em]
EP(z)&=&zP(z),\\[0.3em]
FP(z)&=&\displaystyle-\frac{q^{\lambda-1}(P(q^2z)-P(z))-q^{-\lambda+1}(P(z)-P(q^{-2}z))}{(q-q^{-1})^2z}.
\end{array}
\end{equation}
On the basis $\{z^n\}$ we get:
\begin{equation}\label{repbasisformula}
\begin{array}{rcl}
K^{\pm 1} z^n&=&q^{\pm(\lambda+2n)}z^n,\\[0.3em]
E z^n &=&z^{n+1},\\[0.3em]
F z^n&= &-[\lambda+n-1][n] z^{n-1}.
\end{array}
\end{equation}
One easily checks that this gives a representation of $U_q(sl_2)$ and that the Casimir operator acts as a number, namely, $Cz^n=(q^{\lambda-1}+q^{-\lambda+1})z^n$ for any $n\geq0$. The representation $V_{\lambda}$ is the lowest weight Verma module associated to the weight $\lambda$. It is irreducible for $\lambda\in \C\backslash\Z_{\leq 0}$.

\begin{rema}
An equivalent presentation of $U_q(sl_2)$, used for example in \cite{CFGPRV,Ros}, is found by setting $\tilde{K}=K^{1/2}$, $X_+=E\tilde{K}^{-1}$ and $X_-=\tilde{K}F$. Note that the coproduct becomes more symmetric:
\[\Delta(X_{\pm})=X_{\pm}\otimes \tilde{K}^{-1} + \tilde{K}\otimes X_{\pm}\ .\]
However, we choose the presentation with $K,E,F$ above in order to have in the simplest way a result like in Proposition \ref{propdeltaEanddecomposition} below. 
\end{rema}

\subsection{Tensor product of two Verma modules}

From now on and until the end of the paper we fix $\lambda,\lambda'\in \C\backslash\Z_{\leq 0}$ such that $\lambda+\lambda'\notin \mathbb{Z}_{\leq 0}$. The reason for this condition, as will be shown below, is to have the following decomposition:
\[V_\lambda\otimes V_{\lambda'}\cong \bigoplus_{n\geq0} V_{\lambda+\lambda'+2n}\ .\]
We realise the representation $V_\lambda$ on the space $\C[x]$ and $V_{\lambda'}$ on the space $\C[y]$. We denote the basis elements $x^k\otimes y^l$ simply by $x^ky^l$. On this basis, the action of the generators of $U_q(sl_2)$ is:
\begin{align}
\Delta(K)~x^ky^l&=q^{\lambda+\lambda'+2k+2l}x^ky^l\,,\label{actKxy}\\
\Delta(E)~x^ky^l&=x^{k+1}y^l+q^{\lambda+2k}x^ky^{l+1}\,,\label{actExy}\\
\Delta(F)~x^ky^l&= -q^{-\lambda'-2l}[\lambda+k-1][k] x^{k-1}y^l -[\lambda'+l-1][l] x^ky^{l-1}\ .\label{actFxy}
\end{align}
Note that we do not need for now to specify if $x,y$ are commuting variables or not. In fact, we will see in a moment that it will be useful to impose a $q$-commutation relation between $x$ and $y$, thereby identifying the space $\C[x]\otimes \C[y]$ with a quantum plane.

We start by explicitly finding the lowest weight vectors in the tensor product.

\begin{prop}\label{prop_lowestvectors}
For any $n\geq 0$, the element:
\begin{align}
P_n^{\lambda,\lambda'}(x,y) = [\lambda]_n[\lambda']_n\sum_{k=0}^n \qbinom{n}{k}(-1)^k q^{k(\lambda'+2n-k-1)}\frac{x^ky^{n-k}}{[\lambda]_{k}[\lambda']_{n-k} },
\end{align}
is a non-zero lowest weight vector of weight $\lambda+\lambda'+2n$, i.e. $\Delta(F) P_n^{\lambda,\lambda'}(x,y)=0$ and $\Delta(K) P_n^{\lambda,\lambda'}(x,y)=q^{\lambda+\lambda'+2n}P_n^{\lambda,\lambda'}(x,y)$.
\end{prop}
\begin{proof}
To see that $P^{\lambda,\lambda'}_n(x,y)$ is non-zero, we look at the coefficients in front of $x^n$ and $y^n$. We find respectively, up to signs and powers of $q$, $[\lambda']_n$ and $[\lambda]_n$. The conditions $\lambda,\lambda'\notin \Z_{\leq 0}$ ensure that these coefficients are not zero.

The action of $\Delta(K)$ is found immediately, so it only remains to check that the vector in (\ref{prop_lowestvectors}) is annihilated by $\Delta(F)$. We set $c_k=\frac{[\lambda]_n[\lambda']_n}{[\lambda]_k[\lambda']_{n-k}}$ and $\alpha_k=k(\lambda'+2n-k-1)$.
\begin{align*}
-\Delta(F)P_n^{\lambda,\lambda'}(x,y)&=\sum_{k=0}^n(-1)^k \qbinom{n}{k}c_k~ q^{\alpha_k}([\lambda'+n-k-1][n-k] x^ky^{n-k-1}\\
&\qquad \qquad \qquad \qquad \qquad \quad +q^{-\lambda'-2(n-k)}[\lambda+k-1][k] x^{k-1}y^{n-k} )\\
&=\sum_{k=0}^{n-1}(-1)^k \qbinom{n}{k} c_k~ q^{\alpha_k}[\lambda'+n-k-1] [n-k] x^ky^{n-k-1}\\
&+ \sum_{k=0}^{n-1}(-1)^{k+1} \qbinom{n}{k+1} c_{k+1}~ q^{\alpha_{k+1}-\lambda'-2(n-k-1)}[\lambda+k][k+1]x^ky^{n-k-1} .
\end{align*}
We have:
\begin{equation*}
\qbinom{n}{k}c_k[\lambda'+n-k-1][n-k]= \qbinom{n}{k+1}c_{k+1}[\lambda+k][k+1].
\end{equation*}
Thus
\begin{multline*}
-\Delta(F)P_n^{\lambda,\lambda'}(x,y)=
\sum_{k=0}^{n-1}(-1)^k\qbinom{n}{k}c_k[\lambda'+n-k-1][n-k] x^ky^{n-k-1} (q^{\alpha_k}-q^{\alpha_{k+1}-\lambda'-2(n-k-1)}).
\end{multline*}
Finally, we just check that the two powers of $q$ are equals to get the result. 
\end{proof}

\paragraph{The quantum plane $\C_q[x,y]$.} Now that we have lowest weight vectors, we can construct subrepresentations by repeated applications of $\Delta(E)$. Note that in a Verma module realised on $\C[z]$, the action of $E$ is by multiplication by $z$. We are going now to invoke a $q$-commutation relation between $x$ and $y$ in order to have the action of $\Delta(E)$ also given by multiplication.

From now on, we are going to identify the tensor product $\C[x]\otimes \C[y]$ with the quantum plane $\C_q[x,y]$,  defined as the space of polynomials in $x,y$ such that 
\begin{equation}\label{quantumplane}
yx=q^2xy\ .
\end{equation}
In other words, we put an algebra structure on $\C[x]\otimes \C[y]$ by asking that $yx=q^2xy$. We collect below the first advantage of having this $q$-commutation relation between $x$ and $y$.
\begin{prop}\label{propdeltaEanddecomposition}
The action of $\Delta(E)$ on $\C_q[x,y]$ is the left multiplication by $(x+q^{\lambda}y)$:
\begin{equation}
\Delta(E)~x^ky^l=(x+q^{\lambda}y)x^{k}y^l\,,\label{actExymult}
\end{equation}
and the decomposition 
\begin{equation}\C_q[x,y]\simeq \bigoplus_{n\geq0} \C[x+q^\lambda y] P_n^{\lambda,\lambda'}(x,y)\label{decxy} \end{equation}
gives the decomposition of $V_\lambda\otimes V_{\lambda'}$ into the direct sum of subrepresentations $V_{\lambda+\lambda'+2n}$.

\end{prop}
To be perfectly clear, in order to see an element such as
\[(x+q^{\lambda}y)^kP_n^{\lambda,\lambda'}(x,y)\ \]
as a vector in $\C[x]\otimes \C[y]$, we need to use $yx=q^2xy$ in order to write it in the basis of ordered monomials $x^ky^l$, which is then identified with the standard basis $x^k\otimes y^l$ of $\C[x]\otimes \C[y]$.
\begin{proof}[Proof of Proposition \ref{propdeltaEanddecomposition}] Using the relation $yx=q^2xy$, we have:
\[(x+q^{\lambda}y)x^{k}y^l=x^{k+1}y^l+q^{\lambda+2k}x^ky^{l+1}\ .\]
This reproduces indeed the correct action of $\Delta(E)$ in (\ref{actExy}). From the fact that $P_n^{\lambda,\lambda'}(x,y)$ is a non-zero lowest weight vector of weight $\lambda+\lambda'+2n$ and the form of the action of $\Delta(E)$, we have that the vectors:
\[(x+q^{\lambda}y)^kP_n^{\lambda,\lambda'}(x,y)\,,\ \ \ k\geq 0\ ,\]
form a submodule isomorphic to $V_{\lambda+\lambda'+2n}$. The condition $\lambda+\lambda'\notin \mathbb{Z}_{\leq 0}$ (and $q$ not a root of unity) ensures that the Casimir operator $\Delta(C)$ takes different values on these different subspaces, that is
\[q^{\lambda+\lambda'+n-1}+q^{-(\lambda+\lambda'+n-1)}\neq q^{\lambda+\lambda'+n'-1}+q^{-(\lambda+\lambda'+n'-1)}\ \ \ \ \ \text{if $n\neq n'$.}\]
Therefore these subspaces are linearly independent. Moreover the direct sum exhausts the whole space $\C_q[x,y]$. This can be seen by noting that there is a non-zero term of degree $N$ in $\C[x+q^\lambda y] P_n^{\lambda,\lambda'}(x,y)$ for each $n=0,\dots,N$, which gives the correct dimension $N+1$ for the homogeneous component of degree $N$. This ends the proof of the decomposition (\ref{decxy}).
\end{proof}
\begin{rema}
The necessity of the $q$-commutation relation between $x$ and $y$ should be clear. If the action of $\Delta(E)$ is to be by multiplication by a polynomial, this polynomial has to be the result of $\Delta(E)$ on $1$. This is $x+q^{\lambda} y$. Then the only way to make the action of $\Delta(E)$ coincide with left multiplication by $x+q^{\lambda} y$ is to have $yx=q^2xy$. 

Actually, it is remarkable that such a simple way is possible and of course it is closely related with the fact that $\Delta(E)$ is not equal to $E\otimes 1+1\otimes E$ but is not too far from it. Another way to see it is as follows. 

Due to the PBW basis of $U_q(sl_2)$, the (lowest weight) Verma module $V_{\lambda}$ is naturally identified as a vector space with $\mathbb{C}[E]$. Moreover, this identification is compatible with the action of $E$.

In a similar way, the tensor product $V_{\lambda}\otimes V_{\lambda'}$ of two Verma modules may be identified with $\mathbb{C}[E\otimes 1,1\otimes E]$ as a vector space. For $q=1$, the identification is still compatible with the action of $\Delta(E)$ since $\Delta(E)=E\otimes 1+1\otimes E$. However, this is not the case anymore for $U_q(sl_2)$. Indeed, the space $\mathbb{C}[E\otimes 1,1\otimes E]$ is not stable by multiplication by $\Delta(E)$ since $\Delta(E)=E\otimes 1 + K\otimes E$. 

Thus, the natural identification to make here is
$$V_{\lambda}\otimes V_{\lambda'}\cong \mathbb{C}[E\otimes 1,K\otimes E]\ .$$
The space of polynomials in $E\otimes 1$ and $K\otimes E$ is now stable by multiplication by $\Delta(E)$. From this the $q$-commutation relation appears naturally, since we have:
$$(K\otimes E)(E\otimes 1)=q^{2}(E\otimes 1)(K\otimes E)\ .$$
Note that to identify $V_{\lambda}\otimes V_{\lambda'}\cong\mathbb{C}[x]\otimes \mathbb{C}[y]$ with $\mathbb{C}[E\otimes 1,K\otimes E]$ respecting the action of $\Delta(E)$, we send $x$ to $E\otimes 1$ and $y$ to $q^{-\lambda}K\otimes E$, see Formula (\ref{actExy}).

As a conclusion, the underlying reason for the $q$-commutation relation between $x$ and $y$ is the $q$-commutation relation between the two terms in $\Delta(E)$.
\end{rema}

We collect for later use a reformulation in terms of intertwining operator of the previous results. In the language of \cite{KobaPevznerIII}, the linear map below is a holographic operator.
\begin{coro}\label{coroPsixy}
The linear map $\Psi_n^{\lambda,\lambda'}$ defined by:
\begin{equation}
\Psi_n^{\lambda,\lambda'}\ :\ \C[z]\to \C_q[x,y]\,,\ \ \ \ \ z^k\mapsto (x+q^\lambda y)^kP_n^{\lambda,\lambda'}(x,y)
\end{equation}
is, up to a scalar, the unique intertwining operator between $V_{\lambda+\lambda'+2n}$ and $V_\lambda\otimes V_{\lambda'}$. 
\end{coro}

\section{Little $q$-Jacobi polynomials and decomposition of tensor products}\label{sec3}

In this section, we introduce a linear isomorphism transforming the quantum plane $\C_q[x,y]$ into a more space of polynomials in two variables. As a result, we show how the little $q$-Jacobi polynomials appear in the lowest weight vectors.

\subsection{Little $q$-Jacobi polynomials}

Following \cite{KLS10}, the little $q$-Jacobi polynomials are defined, using the ${}_2\phi_1$ basic hypergeometric function \eqref{eq:HypergeoDef}, by
\begin{align*}
j_n(X;a,b;q)&={}_2\phi_1\left(\begin{array}{c} q^{-2n},\,abq^{2(n+1)}\\ a q^2\end{array};q,q^2X\right)
%\\ &=\sum_{k=0}^n \frac{(q^{-2n};q)_k(abq^{2(n+1)};q)_k}{(aq^2;q)_k(q^2;q)_k} (q^2X)^k\,.
\end{align*}
We use parameters $a=q^{2\alpha}$ and $b=q^{2\beta}$ with $\alpha, \beta \notin \C\backslash \Z_{\leq 0}$, and this becomes in our notations
\begin{equation}\label{def_Jn}
j_n^{\alpha,\beta}(X)=j_n(X;q^{2\alpha},q^{2\beta};q) =\sum_{k=0}^n \frac{[-n]_k[n+\alpha+\beta+1]_k}{[\alpha+1]_k}q^{k(\beta+1)} \frac{X^k}{[k]!}.
\end{equation}
Notice that the range for the parameter $\beta$ is not necessary at this stage but will be necessary for Proposition \ref{prop_formulasRodrigues}. 
\begin{rema}
Taking the limit $q=1$ in (\ref{def_Jn}), we get the classical Jacobi polynomials with parameters $\alpha,\beta$.
\end{rema}
The little $q$-Jacobi polynomials are eigenvectors of the $q$-difference operators $\Theta_X$ given by
\begin{equation}\label{eq:Def_Theta}
\Theta_X f(X)=\frac{q^{-\alpha-\beta-1}}{(q-q^{-1})^2}\left(q^{2\alpha}(q^{2(\beta+1)}X-1)\frac{f(q^2X)-f(X)}{X}-(X-1)\frac{f(X)-f(q^{-2}X)}{X}\right).
\end{equation}
Namely, we have (see  \cite{KLS10}):
\begin{equation}\label{eq:eigenvector_Theta}\Theta_X j_n^{\alpha,\beta}(X)=[\alpha+\beta+n+1][n]j_n^{\alpha,\beta}(X)\ .
\end{equation}

 We will need the following alternative formula for the little $q$-Jacobi polynomials. As we will see in the proof, it follows from the Rodrigues formula.
\begin{prop}\label{prop_formulasRodrigues}
The little $q$-Jacobi polynomials with parameters $\alpha, \beta \notin \C\backslash \Z_{\leq 0}$ satisfy the following formula:

\begin{equation}
j_n^{\alpha,\beta}(X)=[\beta+1]_n\sum_{k=0}^n \qbinom{n}{k}(-1)^k \frac{q^{k(\beta+\alpha+1+k)}}{[\alpha+1]_k[\beta+1]_{n-k}} \left(\prod_{s=0}^{n-k-1}(1-q^{-2s}X)\right)X^k.
\end{equation}

\end{prop}

\begin{proof}
According to \cite{KLS10}, the little $q$-Jacobi polynomials satisfy the following Rodrigues formula:
\begin{equation}\label{Rodrigues}
w(X;\alpha,\beta)j_n^{\alpha,\beta}(X)=c_n(\alpha)  D_{q^{-2}}^n w(X; \alpha+n;\beta+n),
\end{equation}
with
\begin{equation}
w(X;\alpha,\beta)=\left(\prod_{i=0}^\infty \frac{1-q^{2(i+1)}X}{1-q^{2(\beta+i+1)}X}\right)X^\alpha\,, \text{ and } c_n(\alpha)=\frac{q^{n(\alpha+\frac{n-1}{2})}}{[\alpha+1]_n}.
\end{equation}

The operator $D_{q^{-2}}$ satisfies the following $q$-Leibniz formula:
\begin{equation*}
D_{q^{-2}}^n(fg)(X)=\sum_{k=0}^n \qbinom{n}{k} q^{-k(n-k)}f^{(k)}(q^{-2(n-k)}X)g^{(n-k)}(X)\ , 
\end{equation*}
where we use here and below the notation $f^{(k)}=D_{q^{-2}}^{k}(f)$ and $g^{(n-k)}=D_{q^{-2}}^{n-k}(g)$.

Introduce $f_n(X)=\prod_{i=0}^\infty \frac{1-q^{2(i+1)}X}{1-q^{2(\beta+n+i+1)}X}$, and $g_n(X)=X^{\alpha+n}$ so that $w(X; \alpha+n;\beta+n)=f_n(X)g_n(X)$.
Note that
\[f_0(q^{-2(n-k)}X)=\left(\prod_{s=0}^{n-k-1}\frac{1-q^{-2s}X}{1-q^{2(\beta-s)}X}\right)f_0(X)\ \ \ \text{and}\ \ \ f_{n-k}(X)=\left(\prod_{s=1}^{n-k}(1-q^{2(\beta+s)}X)\right)f_0(X)\ .\]
Combining the two equalities and simplifying, we get:
\begin{equation}\label{eqfkqkx}f_{n-k}(q^{-2(n-k)}X)=\left(\prod_{s=0}^{n-k-1}(1-q^{-2s}X)\right)f_0(X)\ .\end{equation}
Now a direct calculation shows that:
\[D_{q^{-2}}(f_n)(X)=-q^{\beta+n+1}[\beta+n]f_{n-1}(X)\ \ \ \ \text{and}\ \ \ \ D_{q^{-2}}(g_n)(X)=q^{-\alpha-n+1}[\alpha+n]g_{n-1}(X)\ .\]
An easy induction on $k$ using also (\ref{eqfkqkx}) gives
\begin{align*}
f_n^{(k)}(q^{-2(n-k)}X)&=(-1)^{k} q^{k(\beta+n+1)-\frac{k(k-1)}{2}}\frac{[\beta+1]_n}{[\beta+1]_{n-k}}\left(\prod_{s=0}^{n-k-1}(1-q^{-2s}X)\right)f_0(X),\\
g_n^{(n-k)}(X)&=q^{-\alpha(n-k)-\frac{n(n-1)}{2}+\frac{k(k-1)}{2}} \frac{[\alpha+1]_n}{[\alpha+1]_{k}}X^{k}g_{0}(X).
\end{align*}
Thus, using the $q$-Leibniz formula in the form given above, we get:
\begin{multline*}
D_{q^{-2}}^n(f_ng_n)(X)=f_0(X)g_0(X) \sum_{k=0}^n \qbinom{n}{k}(-1)^{k} \frac{[\alpha+1]_n[\beta+1]_n}{[\alpha+1]_{k}[\beta+1]_{n-k}}\\
\times q^{k(\alpha+\beta+k+1) -\alpha n-\frac{n(n-1)}{2}}\left(\prod_{s=0}^{n-k-1}(1-q^{-2s}X)\right)X^k.
\end{multline*}
Multiplying by $c_n(\alpha)=q^{\alpha n+\frac{n(n-1)}{2}}[\alpha+1]_n^{-1}$ and dividing by $w(X;\alpha,\beta)=f_0(X)g_0(X)$, we use Rodrigues formula (\ref{Rodrigues}) to obtain the desired formula for $j_n^{\alpha,\beta}(X)$.
\end{proof}
\begin{rema}
Exchanging the role of $f_n$ and $g_n$ in the proof, one gets another formula for the little $q$-Jacobi polynomials:
    \begin{equation*}
    j_n^{\alpha,\beta}(X)=[\beta+1]_n\sum_{k=0}^n \qbinom{n}{k}(-1)^k \frac{1}{[\alpha+1]_k[\beta+1]_{n-k}} q^{k(\beta-\alpha-k+1)}\left(\prod_{s=1}^{n-k}(1-q^{2(\beta+s)}X)\right)X^k.
    \end{equation*}
\end{rema}

\subsection{Little $q$-Jacobi polynomials and lowest weight vectors}

Recall from Proposition \ref{prop_lowestvectors} that the following polynomials in two $q$-commuting variables $x,y$  were found to be the lowest weight vectors in the tensor product $V_{\lambda}\otimes V_{\lambda'}$:
\begin{equation}
P_n^{\lambda,\lambda'}(x,y)=[\lambda]_n[\lambda']_n \sum_{k=0}^n \qbinom{n}{k}(-1)^k  q^{k(\lambda'+2n-k-1)}\frac{x^ky^{n-k}}{[\lambda]_k[\lambda']_{n-k}}.
\end{equation}
Our goal now is to make the connection with the little $q$-Jacobi polynomials.
\begin{rema}
When $q=1$, we have the following relations between the Jacobi polynomials and these two-variable polynomials:
\begin{equation}\label{classicalrel_PJ}
P^{\lambda,\lambda'}_n(x,y)=[\lambda]_n(x+y)^nj^{\alpha,\beta}_n\left(\frac{x}{x+y}\right)\ \ \ \ \text{and}\ \ \ \ j^{\alpha,\beta}_n(X)=\frac{1}{[\lambda]_n}P^{\lambda,\lambda'}_n(X,X-1)\ ,
\end{equation}
where $\alpha+1=\lambda$ and $\beta+1=\lambda'$. Another way of expressing this is through the change of variables $t=x+y$ and $X=\frac{x}{x+y}$. This change of variables identifies $\C[x,y]$ and $\C[t,tX]$, and in this case it is an algebra isomorphism. In the $q$-deformed setting, we can only have a linear isomorphism between $\C_q[x,y]$ and $\C[t,tX]$, and this is what we find below.
\end{rema}

\paragraph{The linear isomorphism $\phi$.} In order to generalise Relations (\ref{classicalrel_PJ}) to the $q$-deformed setting, we are going to construct a linear isomorphism between the quantum plane $\C_q[x,y]$ and the subspace $\C[t,tX]$ of polynomials in two commuting variables $t,X$. We define a linear map $\phi: \C[t,tX] \to \C_q[x,y]$ by:
\begin{equation}\label{defphi}\phi\bigl(t^n(tX)^m\bigr)=(x+q^\lambda y)^n x^m=\sum_{k=0}^n\qbinom{n}{k}q^{(k+\lambda+2m)(n-k)}x^{k+m}y^{n-k}\ .
\end{equation}
For the second equality we used the $q$-binomial formula for variables $x,y$ satisfying $yx=q^2xy$:
\begin{equation}\label{eq:qBinomial}
(x+y)^n=\sum_{k=0}^n \qbinom{n}{k}q^{k(n-k)} x^ky^{n-k}\ .
\end{equation}
\begin{lemm}\label{lemma_actionphi}
The map $\phi$ is invertible and its inverse is given by :
\begin{equation}\label{invphi}
\phi^{-1}(x^ky^p)=q^{-p(\lambda+2k)}t^{k+p}\left(\prod_{s=0}^{p-1}(1-q^{-2s}X)\right) X^k\ .
\end{equation}
\end{lemm}
\begin{proof}
We prove by induction on $p$ that
\begin{equation}\label{invphi_proof}
\phi\left(t^{k+p}\left(\prod_{s=0}^{p-1}(1-q^{-2s}X)\right) X^k\right)=q^{p\lambda}y^px^k.
\end{equation}
The case $p=0$ is clear. Moreover, note that $\phi(tQ)=(x+q^{\lambda}y)\phi(Q)$ for any $Q\in\C[t,tX]$. Using this remark and the induction hypothesis, we make the following calculation:
\begin{align*}
&\phi\left(t^{k+1+p}\left(\prod_{s=0}^p (1-q^{-2s}X)\right) X^k\right) \\
&=\phi\left(t^{k+1+p}\left(\prod_{s=0}^{p-1} (1-q^{-2s}X)\right) X^k\right)-q^{-2p}\phi\left(t^{k+1+p}\left(\prod_{s=0}^{p-1} (1-q^{-2s}X)\right) X^{k+1}\right)\\
&=q^{p\lambda}(x+q^\lambda y)y^p x^k-q^{-2p}q^{p\lambda}y^px^{k+1} =q^{(p+1)\lambda}y^{p+1}x^k.
\end{align*}
Formula (\ref{invphi}) follows from (\ref{invphi_proof}) and $yx=q^2xy$.
\end{proof}

We are ready to state the main property of the map $\phi$, providing the link between the little $q$-Jacobi polynomials and the decomposition of the tensor product $V_{\lambda}\otimes V_{\lambda'}$ of $U_q(sl_2)$-representations. Below, through the map $\phi$ (more precisely, its inverse $\phi^{-1}$), the tensor product is identified to a space of polynomials in two variables $t$ and $tX$ as follows:
\[V_{\lambda}\otimes V_{\lambda'}\,\cong\, \C[x]\otimes\C[y]\,\cong\,\C_q[x,y]\,\stackrel{\phi^{-1}}{\cong}\,\C[t,tX]\ .\]
\begin{coro}\label{coro_phi}
The map $\phi$ satisfies:
\[\phi\ :\ [\lambda]_n t^{n+k}j^{\lambda-1,\lambda'-1}_n(X)\ \mapsto\ q^{n\lambda}(x+q^\lambda y)^kP_n^{\lambda,\lambda'}(x,y)\ .\]
The vector space decomposition 
\[\C[t,tX]\simeq \bigoplus_n \C[t] t^nj^{\lambda-1,\lambda'-1}_n(X) \]
gives the decomposition of $V_\lambda\otimes V_{\lambda'}$ into a direct sum of subrepresentations $V_{\lambda+\lambda'+2n}$.
\end{coro}
\begin{proof}
Using the formula in Proposition \ref{prop_formulasRodrigues} together with Lemma \ref{lemma_actionphi}, it is follows directly that $\phi$ sends $[\lambda]_n t^{n}j^{\lambda-1,\lambda'-1}_n(X)$ to $q^{n\lambda}P_n^{\lambda,\lambda'}(x,y)$. Then we use that, from the definition of $\phi$ we have $\phi(t^kQ)=(x+q^{\lambda}y)^k\phi(Q)$ for any $Q\in\C[t,tX]$.

The decomposition of the representation is the direct translation into the $(t,X)$-picture of Proposition \ref{propdeltaEanddecomposition} giving the decomposition in the $(x,y)$-picture.
\end{proof}

As a consequence, we can also express the intertwining operator (\emph{i.e.} the holographic operator) from Corollary \ref{coroPsixy} in the $(t,X)$-picture.
\begin{coro}\label{coroPsitX}
The map 
\[ \Psi_n^{\lambda,\lambda'}\ :\ \C[z]\to \C[t,tX]\,,\ \ \ \ \ z^k\mapsto t^{k+n}j^{\lambda-1,\lambda'-1}_n(X)\]
is, up to a scalar, the unique intertwining operator between $V_{\lambda+\lambda'+2n}$ and $V_\lambda\otimes V_{\lambda'}$. 
\end{coro}

\subsection{Action of the generators in the new variables}

Now that we have translated, through the map $\phi$ from (\ref{defphi}), the representation of $U_q(sl_2)$ acting on $\C_q[x,y]$ into a representation on $\C[t,tX]$, we can calculate the action of the generators in the new model in $t$ and $X$. The separation of variable mechanism is clearly apparent in the action of $\Delta(K)$ and $\Delta(E)$ below. Indeed the action reproduces on the variable $t$ the original action (\ref{Verma}) of $U_q(sl_2)$ on one-variable polynomials, without touching the variable $X$. There is only an additional term in $\Delta(F)$ involving the $q$-difference operator having the little $q$-Jacobi polynomials as eigenvectors. Therefore, the action below is another way to formulate the decomposition of the tensor product formulated in Corollary \ref{coro_phi}.
\begin{prop}\label{prop_actiontX}
For $P\in \C[t,tX]$ we have:
\begin{align}
\Delta(K)P(t,X)&=q^{\lambda+\lambda'}P(q^2t,X),\\
\Delta(E)P(t,X)&=tP(t,X).\\
\Delta(F)P(t,X)&=-\frac{q^{\lambda+\lambda'-1}(P(q^2t,X)-P(t,X))-q^{-\lambda-\lambda'+1}(P(t,X)-P(q^{-2}t,X))}{(q-q^{-1})^2t} \notag\\
 & \ \ \ \qquad\qquad +\frac{1}{t}\Theta_X P(t,X)\,, \label{eq:DeltaF_en_t,X}
\end{align} 
where $\Theta_X$ is defined in \eqref{eq:Def_Theta} and applied only to the variable $X$. 
\end{prop}
\noindent Note that the given action of $\Delta(F)$ is well-defined on $\C[t,tX]$ as will be seen during the proof.
\begin{proof}
The verification for $K$ and $E$ is an immediate calculation using the formula in $(x,y)$ found in (\ref{actKxy})--(\ref{actFxy}) together with the definition of $\phi$ in (\ref{defphi}).

Let $n\geq 0$ and set $v_k=t^{k+n}j^{\lambda-1,\lambda'-1}_n(X)$ for $k\geq 0$. From Corollary \ref{coroPsitX}, we know that $\{v_k\}_{k\geq 0}$ spans the copy of $V_{\lambda+\lambda'+2n}$ in $V_{\lambda}\otimes V_{\lambda'}$. Moreover, we have $\Delta(K)v_k=q^{\lambda+\lambda'+2n+2k}v_k$ and $\Delta(E)v_k=v_{k+1}$. Therefore, the action of $\Delta(F)$ is uniquely fixed by the relations of $U_q(sl_2)$ and, as in (\ref{repbasisformula}), we have
\begin{equation}\label{proofDeltaF}\Delta(F)v_k=-[\lambda+\lambda'+2n+k-1][k]v_{k-1}\ .\end{equation}
To conclude the proof, we must only check that the right hand side of \eqref{eq:DeltaF_en_t,X} reproduces the correct formula on the vectors $v_k$. Using the same formula (\ref{repbasisformula}) in $t$ for the first term, and the eigenvalue equation \eqref{eq:Def_Theta} for $\Theta_X$, we find
\begin{equation}
(-[\lambda+\lambda'+n+k-1][n+k]+[\lambda+\lambda'+n-1][n])v_{k-1}. 
\end{equation}
Note that the parenthesis is $0$ when $k=0$, which shows that the given action of $\Delta(F)$ preserves $\C[t,tX]$ as it should. Moreover, for arbitrary $k\geq0$, it is easy to check that the parenthesis gives the value expected from (\ref{proofDeltaF}).
\end{proof}

For later use, we also collect the action of the Casimir operator in the $(t,X)$-model. Up to normalisation and up to an additive constant, the element $\Delta(C)$ acts only on $X$ (and does not touch the variable $t$) and reproduces the operator $\Theta_X$ having the little $q$-Jacobi polynomials as eigenvectors. This is once again making apparent the mechanism of separation of variables from $(x,y)$ to $(t,X)$, and the fact that the little $q$-Jacobi polynomials in $X$ provide the decomposition of the tensor product.
\begin{coro}\label{coro_actiondeltaC}
In the variables $(t,X)$ the action of the Casimir operator $C$ \eqref{eq:Casimir} is explicitly given by:
\begin{equation}
\Delta(C) =(q-q^{-1})^2\Theta_X +q^{\lambda+\lambda'-1}+q^{-\lambda-\lambda'+1}.
\end{equation} 
\end{coro}
Note that the eigenvalue of $\Delta(C)$ on $t^{k+n}j^{\lambda-1,\lambda'-1}_n(X)$ is $q^{\lambda+\lambda'+2n-1}+q^{-\lambda-\lambda'-2n+1}$ which is the expected eigenvalue on $V_{\lambda+\lambda'+2n}$.

\section{Little $q$-Jacobi polynomials and Clebsch--Gordan coefficients}\label{sec4}

The purpose of this section is to formulate the Clebsch--Gordan decomposition in terms of the little $q$-Jacobi polynomials. As is well known, the $q$-Hahn polynomials appear as Clebsch--Gordan coefficients for $U_q(sl_2)$. The formulas relating the little $q$-Jacobi polynomials and the $q$-Hahn polynomials that will be obtained in this section are
\begin{equation}\label{convolution1}
    (x+q^{\lambda}y)^{N-k}P^{\lambda,\lambda'}_k(x,y)=[\lambda]_kq^{-\lambda k}\sum_{l=0}^N\qbinom{N}{l}q^{(\lambda+l)(N-l)}Q_{k}^{(N)}(q^{-2l})x^ly^{N-l}\,,\ \ \ \ \ \forall k\leq N\,,
\end{equation}
\begin{equation}\label{convolution2}
    j^{\lambda-1,\lambda'-1}_k(X)=\sum_{l=0}^N\qbinom{N}{l}q^{-l(N-l)}Q_{k}^{(N)}(q^{-2l})\left(\prod_{s=0}^{N-l-1}(1-q^{-2s}X)\right)X^l\,,\ \ \ \ \ \forall k\leq N\ .
\end{equation}
where the polynomials $Q_k^{(N)}$ appearing in the sums are the $q$-Hahn polynomials. The first formula is written in the $(x,y)$-picture (that is, it is valid in the quantum plane where $x$ and $y$ $q$-commute) while the second one is given in the commuting $(t,X)$-picture. They are equivalent through the linear isomorphism $\phi$ from the preceding section.

\subsection{Clebsch--Gordan coefficients and $q$-Hahn polynomials}

We briefly recall how the $q$-Hahn polynomials appear as Clebsch--Gordan coefficients of $U_q(sl_2)$ (see \cite{GranovskiZhedanov92,KVJ98,Ros}) and sketch the construction and the use in this context of the $q$-Hahn algebra \cite{Huang18} (see also \cite{Zhe91}). In this section, we assume that $\lambda,\lambda'\notin\mathbb{Z}$ to avoid possible cancellations in the $q$-Hahn polynomials.

\paragraph{The $q$-Hahn polynomials.} Fix a positive integer $N>0$. Following \cite{KLS10} with $q$ replaced by $q^2$, the q-Hahn polynomials are defined, using the ${}_3\phi_2$ basic hypergeometric function \eqref{eq:HypergeoDef}, by:
\begin{align*}
Q^{(N)}_k(q^{-2x})&={}_3\phi_2\left(\begin{array}{c} q^{-2k},\,abq^{2(k+1)},\,q^{-2x}\\ a q^2,\,q^{-2N}\end{array};q,q^2\right)
% \\ &=\sum_{s=0}^k \frac{(q^{-2k};q)_s(abq^{2(k+1)};q)_s(q^{-2x};q)_s}{(aq^2;q)_s(q^{-2N};q)_s(q^2;q)_s} (q^2)^s\ 
\ \ \ \text{for $k=0,\dots,N$}.
\end{align*}
We use parameters $a=q^{2\alpha}$ and $b=q^{2\beta}$, with $\alpha,\beta\notin\mathbb{Z}$, and this becomes with our notations
\begin{equation}\label{def_Qn}
Q^{(N)}_k(q^{-2x}) =\sum_{i=0}^k \qbinom{k}{i} q^{i(\beta+1+N-x)}(-1)^i\frac{[k+\alpha+\beta+1]_i[-x]_i}{[\alpha+1]_i[-N]_i}.
\end{equation}
Seeing them as polynomials in $X=q^{-2x}$, the $q$-Hahn polynomials satisfy the following $q$-difference equation, for $k=0,\dots,N$,
\begin{equation}\label{diff_qHahn}(q^{-2k}+abq^{2k+2})Q^{(N)}_{k}(X)=B(X)Q^{(N)}_{k}(q^{-2}X)+M(X)Q^{(N)}_{k}(X)+D(X)Q^{(N)}_{k}(q^2X)\,,
\end{equation}
where 
\[\begin{array}{c}B(X)=(1-q^{-2N}X^{-1})(1-a q^{2}X^{-1}),\ \ \ \ D(X)=aq^2(1-X^{-1})(b-q^{-2(N+1)}X^{-1}),\\[0.5em]
M(X)=-B(X)-D(X)+1+abq^2\ .
\end{array}\]
They also satisfy a three-term recurrence relation of the form
\begin{equation}\label{rec_qHahn}XQ^{(N)}_{k}(X)=A_kQ^{(N)}_{k+1}(X)+N_kQ^{(N)}_{k}(X)+C_kQ^{(N)}_{k-1}(X)\,,\ \ \ \ k=0,\dots,N-1\ ,
\end{equation}
where $C_0=0$. We will not need the explicit values of the coefficients $A_k,C_k,N_k$ (see \cite{KLS10}).

\paragraph{The $q$-Hahn algebra.}The $q$-Hahn algebra controls the bispectral properties of the $q$-Hahn polynomials. First one considers the following operators
\[U=X\ \ \ \ \text{and}\ \ \ \ \ V=B(X)T_{q^{-2}}+M(X)+D(X)T_{q^2}\,,\]
where we denote $T_{q^{\pm2}}$ the operators acting as $(T_{q^{\pm2}}f)(X)=f(q^{\pm2}X)$. So $U$ is the left hand side of the recurrence relation (\ref{rec_qHahn}) while $V$ is the $q$-difference operator in (\ref{diff_qHahn}). A straightforward calculation shows that the following relations are satisfied:
\begin{equation}\label{qHahnalgebra}\begin{array}{l}
\displaystyle\frac{[V,U]_q}{q-q^{-1}}=W\,,\\[1em]
\displaystyle\frac{[U,W]_q}{q-q^{-1}}=x_1 U-aq^{-2N}(1+q^2)\ ,\\[1em]
\displaystyle\frac{[W,V]_q}{q-q^{-1}}=x_1 V+ab(1+q^2)^2 U-a(1+q^2)x_2\ .
\end{array}\end{equation}
where the $q$-commutator is $[A,B]_q=qAB-q^{-1}BA$ and with
\[x_1=q^{-2N}(1+a+aq^{2+2N}+abq^{2+2N})\,,\ \ \ \ x_2=q^{-2N}(1+b+bq^{2+2N}+abq^{2+2N})\ .\]
Note that the first relation defines $W$ in terms of $U$ and $V$. Finally the $q$-Hahn algebra is defined as the abstract algebra generated by generators $U,V$ subjected to the defining relations (\ref{qHahnalgebra}).

\paragraph{The $q$-Hahn polynomials as transition coefficients.} The exact same line of arguments as in \cite[\S 3]{LabrietPoulain23} for example (where the classical version ($q=1$) is explained) results in the following connection between the $q$-Hahn algebra and the $q$-Hahn polynomials. We skip the details and only give the results.

Assume that we have two linear operators $U,V$ on an $(N+1)$-dimensional space such that:
\[Sp(U)=\{q^{-2l}\}_{l=0,\dots,N}\ \ \ \ \text{and}\ \ \ \ \text{$U,V$ satisfy the $q$-Hahn algebra.}\]
Then as a consequence we can deduce that $V$ is diagonalisable with 
\[Sp(V)=\{q^{-2k}+abq^{2k+2}\}_{k=0,\dots,N}\ ,\]
and moreover that $U,V$ form a Leonard pair in the sense that $V$ is tridiagonal in the eigenbasis of $U$ and $U$ is tridiagonal in the eigenbasis of $V$. Finally, and most importantly for us, the transition coefficients between the two eigenbases are given in terms of the $q$-Hahn  polynomials. Namely we have:
\[w_k=\sum_{l=0}^N Q^{(N)}_{k}(q^{-2l}) v_l\,,\ \ \ \ \ k,l=0,\dots,N,\]
with $\{v_l\}_{l=0,\dots,N}$, $\{w_k\}_{k=0,\dots,N}$ suitably normalised eigenbases of, respectively, $U$ and $V$.

\paragraph{Clebsch--Gordan coefficients of $U_q(sl_2)$.} From now on, we identify the parameters $a$ and $b$ of the $q$-Hahn algebra in terms of our parameters $\lambda$ and $\lambda'$ as follows:
\begin{equation}\label{param_qHahn}
a=q^{2(\la-1)}\,,\ \ \ b=q^{2(\la'-1)}\ .
\end{equation} 
It turns out that the $q$-Hahn algebra appears in $U_q(sl_2)\otimes U_q(sl_2)$. More precisely, we consider the following two elements of $U_q(sl_2)\otimes U_q(sl_2)$:
\begin{equation}\label{qHahninUqsl2}
U=q^{\lambda}K^{-1}\otimes 1\ \ \ \text{and}\ \ \ V=q^{\lambda+\lambda'-1}\Delta(C)\ .\end{equation}
The normalisation factors are chosen such that in the tensor product
\[V_{\lambda}\otimes V_{\lambda'}\cong\bigoplus_{k\geq0} V_{\lambda+\lambda'+2k}\ ,\]
the spectra of $U$ and $V$ are the correct ones with respect to the $q$-Hahn algebra, namely,
\[Sp(U)=\{q^{-2l}\}_{l=0,\dots,N}\ \ \ \text{and}\ \ \ Sp(V)=\{q^{-2k}+q^{2\lambda+2\lambda'+2k-2}\}_{k=0,\dots,N}\ ,\]
when restricted to the eigenspace of $\Delta(K)=K\otimes K$ with eigenvalue $q^{\lambda+\lambda'+2N}$.

Now a straightforward calculation shows that the elements $U$ and $V$ defined in (\ref{qHahninUqsl2}) satisfy the relations (\ref{qHahnalgebra}) of the $q$-Hahn algebra, again when restricted to the eigenspace of $\Delta(K)=K\otimes K$ with eigenvalue $q^{\lambda+\lambda'+2N}$ (see \cite{Huang18}). These calculations can be done directly in $U_q(sl_2)\otimes U_q(sl_2)$ with the following elements (which commute with $U$ and $V$) specialised as follows:
\[C\otimes 1\mapsto q^{1-\lambda}+q^{\lambda-1}\,,\ \ \ \ 1\otimes C\mapsto q^{1-\lambda'}+q^{\lambda'-1}\,,\ \ \ \ \ \Delta(K)\mapsto q^{\lambda+\lambda'+2N}\ .\]
Therefore, we deduce that the eigenbases of $U$ and $V$ are related by the $q$-Hahn polynomials as explained just before. 

For a given $N$, that is for the given eigenvalue $q^{\lambda+\lambda'+2N}$ of $K\otimes K$, the eigenbases of $U$ and $V$ are respectively,
\begin{equation}\label{twobasis}\{v^{\lambda}_l\otimes v^{\lambda'}_{N-l}\}_{l=0,\dots,N}\ \ \ \text{and}\ \ \ \{v^{\lambda+\lambda'+2k}_{N-k}\}_{k=0,\dots,N}\ ,\end{equation}
where $v^{\lambda}_i$ are the standard basis vectors, as in (\ref{repbasisformula}) for the Verma module $V_{\lambda}$ of $U_q(sl_2)$.

The connection coefficients between these two bases are called the Clebsch--Gordan coefficients, and therefore we have explained how the Clebsch--Gordan coefficients of $U_q(sl_2)$ are given in terms of the $q$-Hahn  polynomials. We will give more explicit formulas in the next section.

\subsection{Little $q$-Jacobi polynomials and $q$-Hahn polynomials}

The two bases involved in the Clebsch--Gordan problem for $U_q(sl_2)$, see (\ref{twobasis}), are in our realisation on $\C_q[x,y]$:
\[\{x^ly^{N-l}\}_{l=0,\dots,N}\ \ \ \ \text{and}\ \ \ \ \ \{(x+q^{\lambda}y)^{N-k}P_k^{\lambda,\lambda'}(x,y)\}_{k=0,\dots,N}\ .\]
Indeed both bases are in the eigenspace of $\Delta(K)$ with eigenvalue $q^{\lambda+\lambda'+2N}$ as required, and the first basis is the eigenbasis of $K\otimes 1$ while the second basis is the eigenbasis of $\Delta(C)$. For the second basis, this follows from the results in Section \ref{sec2}.

From the discussion in the preceding subsection, we know that these two bases are related by the $q$-Hahn polynomials. We obtain the explicit formula, for all $k\leq N$,
\begin{equation}\label{convolution1bis}
    (x+q^{\lambda}y)^{N-k}P^{\lambda,\lambda'}_k(x,y)=[\lambda]_kq^{-\lambda k}\sum_{l=0}^N\qbinom{N}{l}q^{(\lambda+l)(N-l)}Q_{k}^{(N)}(q^{-2l})x^ly^{N-l}\ .
\end{equation}
Note that we had to fix the normalisation factors in front of both sets of vectors. This is easily done by considering the case $k=0$ and then the coefficient in front of $y^N$. We skip the details.

\vskip .2cm
Through our linear isomorphism $\phi$ from Section \ref{sec3}, we can transport these equalities in the $(t,X)$-picture. The variable $t$ disappears and we get:
\begin{equation}\label{convolution2bis}
    j^{\lambda-1,\lambda'-1}_k(X)=\sum_{l=0}^N\qbinom{N}{l}q^{-l(N-l)}Q_{k}^{(N)}(q^{-2l})\left(\prod_{s=0}^{N-l-1}(1-q^{-2s}X)\right)X^l\,,\ \ \ \ \ \forall k\leq N\ .
\end{equation}
Note that for $k=N$, this reproduces in particular the formula found in Proposition \ref{prop_formulasRodrigues}, proving in passing that:
\[Q^{(N)}_N(q^{-2l})=(-1)^lq^{l(\alpha+\beta+N+1)}\frac{[\beta+1]_N}{[\alpha+1]_N[\beta+1]_{N-l}}\ .\]

\subsection{The $q$-Hahn algebra from the little $q$-Jacobi polynomials}

For convenience for the reader specialised in orthogonal polynomials, we switch back to the usual parameters $a=q^{2\alpha}$ and $b=q^{2\beta}$ of the $q$-Hahn polynomials ($\alpha+1=\lambda$ and $\beta+1=\lambda'$).

We have explained above that the two elements $U=q^{\alpha+1}K^{-1}\otimes 1$ and $V=q^{\alpha+\beta+1}\Delta(C)$ of (a certain specialisation of) $U_q(sl_2)\otimes U_q(sl_2)$ satisfy the relations of the $q$-Hahn algebra. Here we use our model with the variables $(t,X)$ for the tensor product of $V_{\lambda}\otimes V_{\lambda'}$ to deduce a more explicit realisation of the $q$-Hahn algebra, involving the $q$-difference operator of the little $q$-Jacobi polynomials.

First, we recall from Corollary \ref{coro_actiondeltaC} that the diagonal Casimir element $\Delta(C)$ acts in our model with variables $(t,X)$ only on the $X$ variable as
\[\Delta(C)=(q-q^{-1})^2\Theta_X+(q^{\alpha+\beta+1}+q^{-\alpha-\beta-1}). \]
where $\Theta_X$ is the $q$-difference operator having the little $q$-Jacobi polynomials as eigenvectors. It can be written as:
\[q^{\alpha+\beta+1}(q-q^{-1})^2\Theta_X=q^{2\alpha}(q^{2(\beta+1)}X-1)\frac{T_{q^2}-1}{X}-(X-1)\frac{1-T_{q^{-2}}}{X}\,,\]
where we denote $T_{q^{\pm2}}$ the operator acting as $(T_{q^{\pm2}}f)(x)=f(q^{\pm2}x)$.

Thus we only need to supplement the explicit action of $\Delta(C)$ with the action of $q^{\alpha+1}K^{-1}\otimes 1$. This is naturally given in the $(x,y)$-picture, but the next result gives its action in the $(t,X)$-picture.

\begin{prop}
The following elements satisfy the relations (\ref{qHahnalgebra}) of the $q$-Hahn algebra:
\[U=q^{-2N}X+(1-X)T_{q^{-2}}\ \ \ \text{and}\ \ \ V=q^{\alpha+\beta+1}(q-q^{-1})^2\Theta_X+q^{2\alpha+2\beta+2}+1\ .\]
\end{prop}
\begin{proof}
We already know that the formula for $V$ gives indeed  the action of $q^{\alpha+\beta+1}\Delta(C)$. So we only need to prove that the formula for $U$ gives the action of $q^{\lambda}K^{-1}\otimes 1$, when restricted to the eigenspace of $K\otimes K$ for the eigenvalue $q^{\lambda+\lambda'+2N}$.

This eigenspace is spanned, in the $(x,y)$-picture by $x^{l}y^{N-l}$, with $l=0,\dots,N$. In the $(t,X)$-picture, using Lemma \ref{lemma_actionphi}, these vectors become proportional to:
\[t^{N}\left(\prod_{s=0}^{N-l-1}(1-q^{-2s}X)\right)X^l\ .\]
Now it is easy to check that the operator $U$ given in the proposition acts diagonally on these vectors, with eigenvalues $q^{-2l}$. Therefore, it coincides with the action of $=q^{\lambda}K^{-1}\otimes 1$, as required.
\end{proof}

\section{Intertwining operators and $q$-Rankin--Cohen brackets}\label{sec5}

Using the same ideas as in \cite{LabrietPoulain24}, we will construct a $q$-analog of the Rankin--Cohen brackets by constructing explicitly the $U_q(sl_2)$-intertwining operators between the contragredient representations  $\left( V_\lambda \otimes V_{\lambda'}\right)^\star$ and $V_{\lambda+\lambda'+2n}^\star$. To do this we will identify the dual space $V_\lambda^\star$ with $V_\lambda$ via the choice of an inner product on $V_\lambda$. Once an inner product is fixed, the action of $U_q(sl_2)$ for the contragredient representation $V_\lambda^\star$ is given for $X\in U_q(sl_2)$ by 
$$\pi_{\lambda}^{\star}(X)=\pi_{\lambda}\bigl(S(X)\bigr)^{\dagger}\,,$$
where $\pi_{\lambda}$ is the original action on $V_{\lambda}$, $S$ is the antipode and $.^\dagger$ denotes the adjoint for the chosen inner product.  

\paragraph{Contragredient representations.} On $V_\lambda\simeq \C[z]$ we define the $q$-Fischer product by:
\[\langle P,Q\rangle=P(D_{q^2})\overline{Q}|_{z=0}\ .\]
Through all this section we will assume that $q$ is a strictly positive real number (different from $1$) to ensure that this formula defines an inner product on $\C[z]$. We recall that $D_{q^2}(z^n)=q^{(n-1)}[n]z^{n-1}$, and so we have on monomials
\[ \langle z^n,z^m\rangle =q^{\frac{n(n-1)}{2}}[n]! \delta_{n,m}.\]
\begin{rema}
This inner product is such that the multiplication by $z$ and the $q$-derivative $D_{q^2}$ are adjoint to each other, i.e. $D_{q^2}^\dagger=m(z)$.  
\end{rema}
With the identification of $V_{\lambda}^{\star}$ with $\mathbb{C}[z]$ made using the inner product above, the action of $U_q(sl_2)$ on $V_{\lambda}^{\star}$ is
\begin{align}
    &KP(z)= q^{-\lambda}P(q^{-2}z), & & Kz^n=q^{-\lambda-2n}z^n,&\label{action_contragredientK}\\
    &EP(z)=-q^{-\lambda-2}\left(D_{q^2}P\right)(q^{-2}z), & \text{or}\qquad\quad & Ez^n=-[n]q^{-\lambda-n-1}z^{n-1},&\label{action_contragredientE}\\
    &FP(z)= q^{\lambda+2}z\frac{q^\lambda P(q^2z)-q^{-\lambda}P(z)}{q-q^{-1}}, & & Fz^n=[\lambda+n]q^{\lambda+n+2}z^{n+1}.& \label{action_contragredientF}
\end{align}
Note that $V_{\lambda}^{\star}$ is a highest weight Verma module of weight $-\lambda$.

The dual space $(V_\lambda\otimes V_{\lambda'})^\star$ of the tensor product is identified with $\C[x]\otimes \C[y]$ using the product of the two inner products i.e.
\[
\langle P_1\otimes Q_1,P_2\otimes Q_2 \rangle=\langle P_1,P_2 \rangle\langle Q_1,Q_2\rangle. 
\]
As a representation of $U_q(sl_2)$, we have
\begin{equation}\label{iso_dual}(V_\lambda\otimes V_{\lambda'})^\star\simeq V_{\lambda'}^\star\otimes V_\lambda^\star\,
\end{equation}
with the trivial isomorphism, where the action of $U_q(sl_2)$ on $V_\lambda$ and $V_{\lambda'}$ is given in (\ref{action_contragredientK})--(\ref{action_contragredientF}), and the tensor product is made with the coproduct $\Delta$ from Section \ref{sec2}.

\paragraph{$q$-Rankin--Cohen brackets.} We define the $q$-analogue of the Rankin--Cohen brackets as
\begin{equation}
    \operatorname{q-RC}^{\lambda,\lambda'}_n=\left(\Psi_n^{\lambda,\lambda'}\right)^\dagger\ :\ (V_\lambda\otimes V_{\lambda'})^\star\to V_{\lambda+\lambda'+2n}^{\star}
\end{equation}
Using the identifications explained above via the inner products, we can see the operator $\operatorname{q-RC}^{\lambda,\lambda'}_n$ as follows
\[\operatorname{q-RC}^{\lambda,\lambda'}_n\ :\ \mathbb{C}[x]\otimes\mathbb{C}[y]\to\mathbb{C}[z]\ .\]
We are now ready to state our final result which gives an explicit formula for the $q$-analogue of the Rankin--Cohen brackets. We recover below formulas similar to the ones found in \cite{Ros} where isomorphic representations are used. The main new feature here is to see these operators as adjoints to the operators $\Psi_n^{\lambda,\lambda'}$.
\begin{theo}\label{thm:qRankinCohenV1}
The map $\operatorname{q-RC}^{\lambda,\lambda'}_n$ is, up to a scalar, the unique intertwining operator between $V_{\lambda'}^\star\otimes V_{\lambda}^\star$ and $V_{\lambda+\lambda'+2n}^\star$, and is given,
for $f\otimes g\in \C[x]\otimes \C[y]$, by
\begin{equation}\label{eq:qRankinCohenV1}
 \operatorname{q-RC}^{\lambda,\lambda'}_n (f\otimes g)(z)=
[\lambda]_n[\lambda']_n\sum_{k=0}^n \qbinom{n}{k}(-1)^k q^{k(\lambda'+2n-k-1)} \frac{f^{(k)}(z)g^{(n-k)}(q^{\lambda+2k}z)}{[\lambda]_k[\lambda']_{n-k}},
\end{equation}
where $f^{(k)}=D_{q^2}^kf$ and $g^{(k)}=D_{q^2}^kg$.
\end{theo}

\begin{proof}
Since $\Psi_n^{\lambda,\lambda'}$ is an intertwining operator between $V_{\lambda+\lambda'+2n}$ and $V_\lambda\otimes V_{\lambda'}$ according to Corollary \ref{coroPsixy}, then it is well-known and straightforward to check that its adjoint is an intertwining operator between the corresponding contragredient representations. The unicity is also immediate. To obtain its explicit expression first notice that 
\begin{equation}
\Psi_n^{\lambda,\lambda'}=r_{P_n^{\lambda,\lambda'}}\circ ev.
\end{equation}
where the polynomials $P_n^{\lambda,\lambda'}$ were given in Proposition \ref{prop_lowestvectors}, the map $ev $ is given by
\[ev \ :\ \C[z]\ni Q \mapsto Q(x+q^\lambda y)\in \C_q[x,y]\ ,\]
and for any element $P\in \C_q[x,y]$, the operator $r_P$ denotes the right multiplication by $P$ in $\C_q[x,y]$. We thus have
\begin{equation}\label{formulaqRC_adjoints} \operatorname{q-RC}^{\lambda,\lambda'}_n=\left(\Psi_n^{\lambda,\lambda'}\right)^\dagger\ :\ (V_\lambda\otimes V_{\lambda'})^\star\to V_{\lambda+\lambda'+2n}^{\star}=ev^{\dagger}\circ (r_{P_n^{\lambda,\lambda'}})^\dagger\ .\end{equation}
The calculations of the adjoints is an exercise in linear algebra. To avoid confusion, we go back to the notation $\C[x]\otimes \C[y]$ instead of $\C_q[x,y]$ since we consider only the vector space structure here. 

Let us treat the map $ev$ first, whose explicit expression is:
\[
ev\ :\ z^n\mapsto \sum_{k=0}^n \qbinom{n}{k}q^{(k+\lambda)(n-k)}x^k\otimes y^{n-k},
\]
We have expanded $(x+q^{\lambda}y)^n$ using the $q$-commutation relation $yx=q^2xy$, and identified $x^ny^m$ with $x^n\otimes y^m$. A direct computation on the basis $x^n\otimes y^m $ proves that
\begin{equation}\label{evdagger}
ev^\dagger \ :\ x^n\otimes y^m\mapsto  q^{\lambda m}z^{n+m}\,,\ \ \ \quad \text{or equivalently,}\quad \ \  ev^\dagger (P\otimes Q) (z)= P(z)Q(q^\lambda z)\ .\end{equation}
To calculate the second adjoint in (\ref{formulaqRC_adjoints}), note first that
\[
r_{P}^\dagger = P(r_x^\dagger,r_y^\dagger).
\]
for any polynomial $P$. Indeed, for $P=x^ky^l$, we have $r_P=r_y^l\circ r_x^k$, and the result follows since the adjoint reverses the composition. 

Therefore it remains only to calculate $r^\dagger_x$ and $r^\dagger_y$. Explicitly, the linear maps on $\C[x]\otimes \C[y]$ are given by
\[
r_x(x^n\otimes y^m)=q^{2m}x^{n+1}\otimes y^{m}\ \ \ \quad\text{and}\quad\ \ \ r_y(x^n\otimes y^m)=x^n\otimes y^{m+1}.
\]
A direct computation shows that:
\[r_x^\dagger \ :\ x^n\otimes y^m\mapsto q^{2m+n-1}[n]x^{n-1}\otimes y^m\ \ \quad\text{and}\quad\ \ 
r_y^\dagger \ :\ x^n\otimes y^m\mapsto q^{m-1}[m]x^{n}\otimes y^{m-1}\ ,
\]
or equivalently, for $P\otimes Q\in \C[x]\otimes \C[y]$
\begin{equation}\label{rxrydagger}
r_x^\dagger (P\otimes Q)(x,y)=D_{q^2} P(x)\otimes Q(q^2y)\ \ \quad\text{and}\quad\ \ 
r_y^\dagger (P\otimes Q)(x,y)=P(x)\otimes D_{q^2} Q(y).
\end{equation}
Finally, one gets formula \eqref{eq:qRankinCohenV1} by using the explicit 
 expansions of the polynomials $P_n^{\lambda,\lambda'}$ in the $x^ny^m$ basis and using Formulas (\ref{rxrydagger}) and (\ref{evdagger}).
\end{proof}

\begin{rema}
 In the limit $q\to 1$, for which the $q$-derivative $D_{q^2}$ becomes the usual derivative, the operators  $\operatorname{q-RC}^{\lambda,\lambda'}_n$ become the classical Rankin--Cohen brackets \cite{LabrietPoulain24}.
\end{rema}

\paragraph{Adjoints and symbols.} A crucial step of the calculation above is when calculating the adjoint of the operator denoted $r_P$, which is the right multiplication in $\C_q[x,y]$ by a polynomial $P(x,y)$. As mentioned during the proof, the action of $r^{\dagger}_P$ is given on $\C[x]\otimes \C[y]$ by:
\begin{equation}\label{rpdagger}
r^{\dagger}_P=P(r_x^\dagger,r_y^{\dagger})\ .\end{equation}
So in words, the action of $r_P^{\dagger}$ is obtained by taking $P$ and replacing $x$ by $r_x^\dagger$ and $y$ by $r_y^{\dagger}$. The fact expressed in (\ref{rpdagger}) is valid whatever the inner product is, but of course the explicit expressions for $r_x^\dagger$ and $r_y^\dagger$ depend on the chosen inner product.

With our inner product, moving between a polynomial $P(x,y)$ and the operator $P(r_x^\dagger,r_y^{\dagger})$ is a $q$-deformation of the notion of symbol of a differential operator. Indeed, when $q=1$, then $r_x^\dagger$ becomes the usual derivative with respect to $x$ and similarly for $y$, namely, we have
\[(q=1)\qquad\ \ \ \ \ r^{\dagger}_P=P\left(\frac{\partial}{\partial x},\frac{\partial}{\partial y}\right)\ .\]
For the $q$-deformed setting, denote again $T_{q^2}$ the operator on $\C[x]$ given by $T_{q^2}f(x)=f(q^2x)$. Then the explicit formula for $r_P^{\dagger}$ is:
\begin{equation}\label{qsymbol1}
r_P^{\dagger}=P(D_{q^2}\otimes T_{q^2},1\otimes D_{q^2}).\end{equation}
Note that this is well-defined in the sense that we evaluate $P\in \C_q[x,y]$ in two operators which $q$-commute, i.e.
\[(1\otimes D_{q^2})(D_{q^2}\otimes T_{q^2})=q^2(D_{q^2}\otimes T_{q^2})(1\otimes D_{q^2})\ .\]

\begin{rema}
We insist that all the above, including the explicit formula (\ref{eq:qRankinCohenV1}) depends on the choice of inner products on the various spaces under consideration. To generalise naturally the case $q=1$, we require that these inner products go when $q=1$ to the usual Fischer inner product for which the multiplication by $x$ is adjoint to the derivative.

In the following, we mention another interesting choice of inner product. 
On the one-variable polynomial space, we choose the same inner product:
\[\langle P(z),Q(z)\rangle=\left.P(D_{q^2})\overline{Q}\right|_{z=0}\]
However on the quantum plane $\C_q[x,y]$ we define the following inner product:
    \[ 
    \langle     P(x,y),Q(x,y)\rangle'=\left.P(D_{q^{2},x},D_{q^{2},y})\overline{Q(x,y)}\right|_{x=y=0}\,, 
    \]
where we define operators $D_{q^{2},x}$ and
$D_{q^{2},y}$ on $\C_q[x,y]$ as follows
\begin{equation}\label{actionCqxy}
D_{q^{2},x}x^ny^m=q^{n-1}[n]x^{n-1}y^m\ \ \ \quad\text{and}\quad\ \ \ D_{q^{2},y}x^ny^m=q^{m-1-2n}[m]x^{n}y^{m-1}\ .\end{equation}
These formulas says that $D_{q^{2},x}$ and $D_{q^{2},y}$ act as usual on their respective variables. However the action of $D_{q^{2},y}$ on $x^ny^m$ is understood as follows: first move the $y^m$ to the left, then differentiate and then move back the remaining $y^{m-1}$ to the right using the $q$-commutation relation.

A word of warning here: this inner product is not the same as the product of one-variable inner products. In particular, the isomorphism relating the $U_q(sl_2)$-representations $(V_\lambda\otimes V_{\lambda'})^\star$ and $V_{\lambda'}^\star\otimes V_\lambda^\star$ will not be the trivial isomorphism (but will be a renormalisation of the product basis of $V_\lambda\otimes V_{\lambda'}$).

One motivation to consider this alternative is that for $P\in \C_q[x,y]$ we can interpret the symbol of $r_{P}^\dagger$ in a similar fashion as in the classical setting since now we will have:
\[
r_P^\dagger =P(D_{q^{2},x},D_{q^{2},y}).
\] 
This formula may look more natural than (\ref{qsymbol1}) but one has to remember to act on $\C_q[x,y]$ according to Formulas (\ref{actionCqxy}). Note again that  the evaluation above in $r_P^{\dagger}$ is well-defined since $D_{q^{2},x}$ and $D_{q^{2},y}$ $q$-commute when acting on $\C_q[x,y]$.

These choices lead to an operator $\left(\Psi_n^{\lambda,\lambda'}\right)^\dagger$ acting on the quantum plane $\C_q[x,y]$ with value in $\C[z]$ whose explicit formula will be different from \eqref{eq:qRankinCohenV1} since different identifications are made. We leave the interested reader to work out the explicit formula.
\end{rema}
    
\begin{rema}
As a final remark, we note that several possibilities exist for the $q$-derivative. We have chosen to work with $D_{q^2}$ but one could have made a different choice, such as $D_{q^{-2}}$ or $D_q$. The resulting explicit formulas will differ by powers of $q$.
\end{rema}

%%%%%%%%%%%%%%%%%
%%%%%%%%%%%%%%%%%
%%%%%%%%%%%%%%%%%
%%%%%%%%%%%%%%%%%


\begin{thebibliography}{99}

\bibitem{BMVZ20} P. Baseilhac, X. Martin, L. Vinet, A. Zhedanov, \emph{Little and big q-Jacobi polynomials and the Askey--Wilson algebra}.
Ramanujan J. 51, No. 3, 629--648 (2020). 

\bibitem{Cohen75} H. Cohen, \emph{Sums involving the values at negative integers of L-functions of
quadratic characters}, Math. Ann., 217:271–285, 1975.

\bibitem{CFGPRV} N. Crampé, L. Frappat, J. Gaboriaud, L. Poulain d'Andecy, E. Ragoucy, L. Vinet, \emph{The Askey--Wilson algebra and its avatars}, J. Phys. A54(2021), no.6, Paper No. 063001, 32 pp.

\bibitem{GranovskiZhedanov92} Y. I. Granovskiĭ and A. S. Zhedanov, \emph{“Twisted” Clebsch--Gordan coefficients for $SU_q(2)$}, J.
Phys. A 25 (1992), no. 17, L1029–L1032. 

\bibitem{Huang18} H. Huang, \emph{An algebra behind the Clebsch--Gordan coefficients of $U_q(sl_2)$}, J. Algebra 496, 61-90 (2018). 

\bibitem{KKP16} T. Kobayashi, T. Kubo, and M. Pevzner, \emph{Conformal symmetry breaking operators for differential forms
on spheres}, volume 2170. Singapore: Springer, 2016.

\bibitem{KobaPevznerI} T. Kobayashi, M. Pevzner, \emph{Differential symmetry breaking operators. I: General theory and F-method}. Sel. Math., New Ser. 22, No. 2, 801-845 (2016). 

\bibitem{KobaPevznerII}  T. Kobayashi, M. Pevzner,\emph{ Differential symmetry breaking operators. II: Rankin--Cohen operators for symmetric pairs}. Sel. Math., New Ser. 22, No. 2, 847-911 (2016). 

\bibitem{KobaPevznerIII} T. Kobayashi, M. Pevzner, \emph{Inversion of Rankin--Cohen operators via holographic transform.} 
Ann. Inst. Fourier (Grenoble) 70 (2020), no. 5, 2131--2190.

\bibitem{KobaSpehI} T. Kobayashi, B. Speh, \emph{Symmetry breaking for representations of rank one orthogonal groups}, volume 1126 of Mem. Am. Math. Soc. Providence, RI: American Mathematical Society (AMS), 2015.

\bibitem{KobaSpehII} T. Kobayashi, B. Speh,\emph{ Symmetry breaking for representations of rank one orthogonal groups. II},
volume 2234. Singapore: Springer, 2018.

\bibitem{KLS10} R. Koekoek, P.A. Lesky and R.F. Swarttouw, \emph{Hypergeometric orthogonal polynomials and
their q-analogues}, Springer Monographs in Mathematics, Springer Berlin, Heidelberg (2010).

\bibitem{KVJ98} H.T. Koelink, J. Van Der Jeugt, \emph{Convolutions for orthogonal polynomials from Lie and quantum algebra representations}, SIAM J. Math. Anal. 29 (1998), no. 3, 794--822.

\bibitem{Koo90} T.H. Koornwinder, \emph{Orthogonal polynomials in connection with quantum groups}, Orthogonal polynomials (Columbus, OH, 1989), 257--292. NATO Adv. Sci. Inst. Ser. C: Math. Phys. Sci., 294 Kluwer Academic Publishers Group, Dordrecht, 1990.

\bibitem{KuboOrsted24} T. Kubo and B. Ørsted, \emph{On the intertwining differential operators from a line bundle to a vector bundle over the real projective space}, Indag. Math., New Ser. 36, No. 1, 270--301 (2025). 


\bibitem{LabrietPoulain23} Q. Labriet, L. Poulain d'Andecy, \emph{Realisations of Racah algebras using Jacobi operators and convolution identities}, Adv. in Appl. Math. 153 (2024), Paper No. 102620, 31 pp.

\bibitem{LabrietPoulain24} Q. Labriet, L. Poulain d'Andecy, \emph{Identities for Rankin--Cohen brackets, Racah coefficients and associativity}. Lett. Math. Phys. 114 (2024), no. 1, Paper No. 20, 17 pp.

\bibitem{P-V23} V. Pérez-Valdés, \emph{Conformally covariant differential symmetry breaking operators for a vector bundle of
rank $3$ on $S^3$}. International Journal of Mathematics, 34(12), 2023.

\bibitem{Ros} H. Rosengren, \emph{Multivariable q-Hahn polynomials as coupling coefficients for quantum algebra representations}, Int. J. Math. Math. Sci.28(2001), no.6, 331--358.

\bibitem{Zhe91} A. S. Zhedanov, \emph{``Hidden symmetry'' of Askey--Wilson polynomials}, (Russian) Teoret. Mat. Fiz. 89 (1991), no. 2, 190–204; translation in Theoret. and Math. Phys. 89 (1991), no. 2, 1146--1157.


\end{thebibliography}
\end{document}